\documentclass[12pt]{article}
\usepackage{amsmath,amssymb,amsthm,comment,cite}

\newtheorem{theorem}{Theorem}[section]
\newtheorem{lemma}[theorem]{Lemma}

\def\barr{\begin{array}}
\def\earr{\end{array}}

\title{On a group-theoretical generalization of the Euler's totient function}
\author{Marius T\u arn\u auceanu}
\date{October 26, 2021}

\begin{document}

\maketitle

\begin{abstract}
Let $G$ be a finite group and $\varphi(G)=|\{a\in G \mid o(a)=\exp(G)\}|$, where $o(a)$ denotes the order of $a$ in $G$ and $\exp(G)$ denotes
the exponent of $G$. Under a natural hypothesis, in this note we determine the groups $G$ such that $\forall\, H,K\leq G$, $H\subseteq K$
implies $\varphi(H)\mid\varphi(K)$. This partially answers Problem 5.4 in \cite{6}.
\end{abstract}
\smallskip

{\small
\noindent
{\bf MSC 2020\,:} Primary 20D60, 11A25; Secondary 20D99, 11A99.

\noindent
{\bf Key words\,:} Euler's totient function, finite group, order of an element, exponent of a group, nilpotent group, Schmidt group.}

\section{Introduction}

Given a finite group $G$, we consider the function
\begin{equation}
\varphi(G)=|\{a\in G \mid o(a)=\exp(G)\}|\nonumber
\end{equation}introduced in our previous paper \cite{6}. Since $\varphi(\mathbb{Z}_n)=\varphi(n)$, for all $n\in\mathbb{N}^*$, it generalizes the well-known Euler's totient function. In \cite{7}, the function $\varphi$ was used to provide a nilpotency criterion for finite groups, namely
\begin{equation}
G \mbox{ is nilpotent if and only if } \varphi(S)\neq 0 \mbox{ for any section } S \mbox{ of } G.\nonumber
\end{equation}This has been extended to a $p$-nilpotency criterion by A.D. Ramos and A. Viruel \cite{3}. We recall that the weaker condition
\begin{equation}
\varphi(S)\neq 0 \mbox{ for any subgroup } S \mbox{ of } G
\end{equation}does not guarantee the nilpotency of $G$, as it is shown in \cite{7}.\newpage

In the current note, we will describe finite groups $G$ such that $\forall\, H,K\leq G$, $H\subseteq K$ implies $\varphi(H)\mid\varphi(K)$. To avoid the case in which $\varphi(H)=0$, we will assume the condition (1) to be satisfied. Thus, we partially answer Problem 5.4 in \cite{6}.

Our main result is stated as follows.

\begin{theorem}
Let $G$ be a finite group satisfying the condition (1). Then
\begin{equation}
\forall\, H,K\leq G, H\subseteq K \mbox{ implies } \varphi(H)\mid\varphi(K)
\end{equation}if and only if $G$ is nilpotent and its Sylow subgroups are cyclic, $Q_8$ or $\mathbb{Z}_p\times\mathbb{Z}_p$ with $p$ prime.
\end{theorem}

Next we recall two important classes of finite groups that will be used in our note:
\begin{itemize}
\item[$\bullet$]A \textit{generalized quaternion $2$-group} is a group of order $2^n$ for some positive integer $n\geq 3$, defined by the presentation
\begin{equation}
Q_{2^n}=\langle a,b \mid a^{2^{n-2}}= b^2, a^{2^{n-1}}=1, b^{-1}ab=a^{-1}\rangle.\nonumber
\end{equation}These groups are the unique finite non-cyclic $p$-groups all of whose abelian subgroups are cyclic, or equivalently the unique finite non-cyclic $p$-groups possessing exactly one subgroup of order $p$ (see e.g. (4.4) of \cite{5}, II).
\item[$\bullet$]A \textit{Schmidt group} is a non-nilpotent group all of whose proper subgroups are nilpotent. By \cite{4} (see also \cite{1,2}), such a group $G$ is solvable of order $p^mq^n$ (where $p$ and $q$ are different primes) with a unique Sylow $p$-subgroup $P$ and a cyclic Sylow $q$-subgroup $Q$, and hence $G$ is a semidirect product of $P$ by $Q$. Moreover, we have:
\begin{itemize}
\item[-] if $Q=\langle y\rangle$, then $y^q\in Z(G)$;
\item[-] $Z(G)=\Phi(G)=\Phi(P)\times\langle y^q\rangle$, $G'=P$, $P'=(G')'=\Phi(P)$;
\item[-] $|P/P'|=p^r$, where $r$ is the order of $p$ modulo $q$;
\item[-] if $P$ is abelian, then $P$ is an elementary abelian $p$-group of order $p^r$ and $P$ is a minimal normal subgroup of $G$;
\item[-] if $P$ is non-abelian, then $Z(P)=P'=\Phi(P)$ and $|P/Z(P)|=p^r$.
\end{itemize}We mention that $G/Z(G)$ is also a Schmidt group of order $p^rq$ which can be written as semidirect product of an elementary abelian $p$-group of order $p^r$ by a cyclic group of order $q$, and that it does not have cyclic subgroups of order $pq$.
\end{itemize}

Most of our notation is standard and will not be repeated here. Basic definitions and results on group theory can be found in \cite{5}.

\section{Proof of Theorem 1.1}

First of all, we prove two auxiliary results.

\begin{lemma}
Let $G$ be a finite $p$-group satisfying the condition (2). Then $G$ is cyclic, $Q_8$ or $\mathbb{Z}_p\times\mathbb{Z}_p$.
\end{lemma}

\begin{proof}
Assume first that $G$ is abelian. If it is not cyclic, then there exists $H\leq G$ such that $H\cong\mathbb{Z}_p^2$. Let $K$ be a subgroup of order $p^3$ of $G$ which contains $H$. Then either $K\cong\mathbb{Z}_p^3$ or $K\cong\mathbb{Z}_p\times\mathbb{Z}_p^2$. It follows that $\varphi(H)=p^2-1$ divides either $\varphi(\mathbb{Z}_p^3)=p^3-1$ or $\varphi(\mathbb{Z}_p\times\mathbb{Z}_{p^2})=p^2(p-1)$, a contradiction. Thus such a subgroup $K$ does not exist, showing that either $G$ is cyclic or $G=H\cong\mathbb{Z}_p^2$.

Assume now that $G$ is not abelian. Since the condition (2) is inherited by subgroups, we infer that all abelian subgroups of $G$ are either cyclic or of type $\mathbb{Z}_p^2$. Suppose that $G$ has an abelian subgroup $H\cong\mathbb{Z}_p^2$. Then $H$ is contained in a subgroup $K$ of order $p^3$ of one of the following types:
\begin{itemize}
\item[-] $\mathbb{Z}_p^3$;
\item[-] $\mathbb{Z}_p\times\mathbb{Z}_{p^2}$;
\item[-] $M(p^3)=\langle x,y \mid x^{p^2}=y^p=1, y^{-1}xy=x^{p+1}\rangle$;
\item[-] $E(p^3)=\langle x,y \mid x^p=y^p=[x,y]^p=1, [x,y]\in Z(E(p^3))\rangle$;
\item[-] $D_8$.
\end{itemize}It is easy to see that in all cases $\varphi(H)$ does not divide $\varphi(K)$, contradicting our hypothesis. Consequently, all abelian subgroups of $G$ are cyclic, implying that $G\cong Q_{2^n}$ for some $n\geq 3$. If $n\geq 4$, then $G$ has a subgroup of type $Q_8$ and so
\begin{equation}
6=\varphi(Q_8)\mid\varphi(G)=2^{n-2},\nonumber
\end{equation}a contradiction. Hence $G\cong Q_8$, as desired.
\end{proof}

\begin{lemma}
Let $G$ be a finite group satisfying the conditions (1) and (2). Then $G$ is nilpotent.
\end{lemma}

\begin{proof}
Assume that $G$ is a counterexample of minimal order. Then all proper subgroups of $G$ also satisfy the conditions (1) and (2), and therefore $G$ is a Schmidt group. Suppose that it has the structure described in Introduction. By Lemma 2.1, we distinguish the following three cases:
\begin{itemize}
\item[{\rm a)}] $P$ is cyclic
\end{itemize}Then $\exp(G)=p^mq^n$ and so the condition $\varphi(G)\neq 0$ implies that $G$ is cyclic, a contradiction.
\begin{itemize}
\item[{\rm b)}] $P\cong\mathbb{Z}_p\times\mathbb{Z}_p$ with $p$ prime
\end{itemize}Then $Z(G)=\langle y^q\rangle$ and $\exp(G)=pq^n$. Since $\varphi(G)\neq 0$, there exists a cyclic subgroup $M\leq G$ of order $pq^n$. Note that we have $Z(G)\subset M$. It follows that the Schmidt group $G/Z(G)$ of order $p^2q$ contains the cyclic subgroup $M/Z(G)$ of order $pq$, a contradiction.
\begin{itemize}
\item[{\rm c)}] $P\cong Q_8$
\end{itemize}We have $|{\rm Aut}(Q_8)|=24$. Since $G$ is a non-trivial semidirect product of $Q_8$ and $\mathbb{Z}_{q^n}$, we get $q=3$ and so $G\cong Q_8\rtimes\mathbb{Z}_{3^n}$. Similarly with b), we obtain that the Schmidt group $G/Z(G)\cong A_4$ has a cyclic subgroup of order $6$, a contradiction.

Hence such a group $G$ does not exist, as desired.
\end{proof}

We are now able to prove our main result.

\bigskip\noindent{\bf Proof of Theorem 1.1.} The direct implication follows from Lemmas 2.1 and 2.2. Conversely, we observe that it suffices to prove the condition (2) for cyclic $p$-groups, $Q_8$ or $\mathbb{Z}_p\times\mathbb{Z}_p$ with $p$ prime because a nilpotent group is the direct product of its Sylow subgroups. This is obvious, completing the proof.\qed

\vspace*{3ex}\small

\hfill
\begin{minipage}[t]{5cm}
Marius T\u arn\u auceanu \\
Faculty of  Mathematics \\
``Al.I. Cuza'' University \\
Ia\c si, Romania \\
e-mail: {\tt tarnauc@uaic.ro}
\end{minipage}

\end{document}